\documentclass{amsart}
\usepackage{amsmath}
\usepackage{amsthm}
\usepackage{amssymb}
\usepackage{color}
\usepackage{enumerate}
\usepackage[colorlinks]{hyperref}

\renewcommand{\L}{\mathcal{L}}
\renewcommand{\P}{\mathcal{P}}
\newcommand{\R}{\mathbb{R}}

\newtheorem{theorem}{Theorem}
\newtheorem{defi}[theorem]{Definition}
\newtheorem{thm}[theorem]{Theorem}
\newtheorem{rmk}[theorem]{Remark}
\newtheorem{lemma}[theorem]{Lemma}
\newtheorem{coroll}[theorem]{Corollary}
\newtheorem{prop}[theorem]{Proposition}

\newtheorem{notat}[theorem]{Notation}

\newcommand{\EE}{\mathbb{E}}
\newcommand{\RR}{\mathbb{R}}
\newcommand{\PP}{\mathbb{P}}

\newcommand{\bb}{\color{black}}
\newcommand{\dd}{\color{black}}

\newcommand{\Hpi}{H_{\phi}^{-1}}
\newcommand{\vpe}{\varphi_{\epsilon}}
\newcommand{\indiq}{{\bf 1}}
\newcommand{\e}{\epsilon}

\newcommand\blfootnote[1]{%
  \begingroup
  \renewcommand\thefootnote{}\footnote{#1}%
  \addtocounter{footnote}{-1}%
  \endgroup
}

\title{On subexponential convergence to equilibrium of Markov processes}
\author{Armand Bernou}
\address{Sorbonne Universit\'e, CNRS, Laboratoire de Probabilit\'e, Statistique et Mod\'elisation, F-75005 Paris, France.}
\email{armand.bernou@sorbonne-universite.fr}

\begin{document}



\begin{abstract}
Studying the subexponential convergence towards equilibrium of a strong Markov process,
we exhibit an intermediate Lyapunov condition equivalent to the control of some moment of a hitting time.
This provides a link, similar (although more intricate) to the one existing in the exponential case, between
the coupling method and the approach based on the existence of a Lyapunov function for the generator,
in the context of the subexponential rates found by \cite{FortRoberts}, \cite{DFG} and \cite{Hairer}.
\end{abstract} 

\maketitle

\textbf{Acknowledgements:} I would like to acknowledge Nicolas Fournier (LPSM, Sorbonne Université) for  all  the  fruitful  discussions  and  advices  he offered me while preparing this note. The author warmly thanks the anonymous referee for their suggestions and remarks. This work was supported by grants from R\'egion \^Ile de France.

\section{Introduction}

\subsection{Context and main result}
The study of the convergence towards an invariant measure of continuous-time Markov processes has generated a large literature devoted to the geometric case (also referred to as the exponential case).
Meyn and Tweedie and coauthors \cite{Meyn93b, MTContinuous2, DownMT} developed stability
concepts for continuous-time Markov processes along with simple criteria for non-explosion, Harris-recurrence, positive Harris-recurrence, ergodicity and geometric
convergence to equilibrium. When applying those stability concepts, the key question of the existence of verifiable conditions emerges. In the discrete-time context, development of Foster-Lyapunov-type conditions on the transition kernel has provided such criteria. In the continuous-time context, Foster-Lyapunov inequalities applied to the (extended)
generator of the process play the same role. 
One of the key results of this theory is the equivalence of two conditions, both implying an exponential convergence towards equilibrium: the control of the moment of the hitting time of a set with appropriate properties, which can be seen as the conditions necessary to apply a coupling method, and the existence of some test function satisfying a Foster-Lyapunov inequality with respect to the generator.
Loosely speaking, considering a topological space $E$ and a $E$-valued strong Markov process $(X_t)_{t \geq 0}$,
with semigroup $(\P_t)_{t \geq 0}$,
invariant probability distribution $\pi$ and with appropriate properties (irreducibility,
non-explosion and aperiodicity, see Section \ref{Setting} for precise definitions), we have the following result.
Roughly, a set $C \in \mathcal{B}(E)$, where we write $\mathcal{B}(A)$ for the Borel $\sigma$-algebra on the topological space $A$, is said to be {\it petite} if there is a probability measure $a$ on
$\mathcal{B}(\RR_+)$ and a non-trivial measure $\nu$ on $\mathcal{B}(E)$ such that
$ \forall x \in C, \int_0^{\infty} \P_t(x,\cdot) a(dt) \geq \nu(\cdot)$.

\begin{thm}[Exponential case, \cite{SurveyMeynTweedie}]
\label{ThmExpo}
Assume that $(X_t)_{t \geq 0}$ is non-explosive, irreducible,
and aperiodic. Then 
the following conditions are equivalent.
\begin{enumerate}
\item[1.] There exist a closed petite set $C \in \mathcal{B}(E)$ and some constants
$\delta > 0$ and $\kappa > 1$ such that, setting
\[ \tau_C(\delta) = \inf \{t > \delta, X_t \in C\}, \]
we have
\begin{align}
\label{ConditionMomentExpo}
\sup_{x \in C} \EE_x[\kappa^{\tau_C(\delta)}] < \infty. 
\end{align}
\item[2.] There exist a closed petite set $C \in \mathcal{B}(E)$, some constants $b>0$,
$\beta>0$ and $V: E \to [1, \infty]$ finite at some $x_0 \in E$ such that, in the sense of Notation \ref{NotatLf},
\begin{align}
\label{ConditionLyapunovExpo2}
\mathcal{L} V \leq - \beta V + b \mathbf{1}_C.
\end{align}
\end{enumerate} 
Any of those conditions implies that the set $S_V = \{x: V(x) < \infty\}$ is absorbing and full (see Section \ref{Setting} for the precise definitions) for any $V$ solution of (\ref{ConditionLyapunovExpo2}), and that there exists $\rho < 1$ and $d>0$
such that for all $x \in E$,
$$ \|\P_t(x,\cdot) - \pi(\cdot)\|_{TV} \leq d V(x) \rho^t. $$
\end{thm}
%

In the study of subgeometric rates, the situation is quite different. While a moment condition of some hitting time similar to (\ref{ConditionMomentExpo}) can be found, as well as a Lyapunov condition similar to (\ref{ConditionLyapunovExpo2}), there is no equivalence between them. In this note, we identify an intermediate Lyapunov condition, equivalent to some moment condition for subgeometric convergence rates, and prove the following result, with the same notations as above. 

\begin{thm}
\label{MainThm}
Assume that $(X_t)_{t \geq 0}$ is non-explosive, irreducible and aperiodic. Let \bb $\phi$ be a $C^1$ function, with $\phi: [1,\infty) \to [1,\infty)$, \dd strictly increasing, strictly concave with $\phi(x) \le x$ for all $x \geq 1$ and $\frac{\phi(x)}{x} \downarrow 0$, $\phi(x) - x \phi'(x) \uparrow \infty$ when $x \to \infty$. Define the function $H_{\phi}(\cdot)$ on $[1,\infty)$ by
\[ H_{\phi}(u) = \int_{1}^u \frac{ds}{\phi(s)}, \]
and let $\Hpi: [0,\infty) \to [1,\infty)$ be its inverse function. 
Consider the three following conditions.
\begin{enumerate}
\item[(C1)] \label{C1} There exists a compact petite subset $C$ of $E$ and some $r > 0$ such that, for $\tilde{\tau}_C^r$
defined by
\begin{align}
\label{DefTildeTauC}
\tilde{\tau}_C^r = \inf \Big\{t > 0, \int_0^t \mathbf{1}_C(X_s) ds \geq \frac{T}{r}\Big\}, 
\end{align} 
where $T$ is an exponential random variable with parameter 1 independent of everything else,
we have
\begin{align}
\label{ConditionMomentMainThm}
  \EE_x[\Hpi(\tilde{\tau}_C^r)] < \infty \quad \hbox{for all } x \in E \quad \hbox{ and }
  \quad \sup_{x \in C} \EE_x[\Hpi(\tilde{\tau}_C^r)] < \infty.
\end{align}
\item[(C2)] \label{C2} There exists a compact petite subset $C$ of $E$, \bb two constants $\kappa \ge 2$, \dd$ \eta >0$ and a function
$\psi$ on  $\RR_+ \times E$ with values in  $[1,\infty)$, continuous and non-decreasing in its first argument,
continuous in its second argument,
such that for all $t \geq 0$, $x \in E$,
\bb
\begin{align*}
\Hpi(t) &\leq \psi(t,x) \leq \Hpi(t) \psi(0,x), \\
(\partial_t + \mathcal{L}) \psi(t,x) &\leq \kappa \phi(\Hpi(t)) \mathbf{1}_{C}(x) - \phi (\Hpi(t)),
\end{align*}
\dd 
with moreover $\psi(0,x) \leq \kappa$ for all $x\in C$ and  for all $x \in E$, 
$ \mathcal{L} \psi(0,x) \leq \kappa \mathbf{1}_C(x) - \eta$.
\item[(C3)] \label{C3} There exists a compact petite subset $C$ of $E$, a constant $K>0$ and
$V: E \to [1,\infty)$ continuous with precompact sublevel sets such that for all $x\in E$,
\begin{align}
\label{ConditionLyapSubGeom}
\mathcal{L}V(x) \leq -\phi(V(x)) + K \mathbf{1}_C(x).
\end{align} 
\end{enumerate}
Conditions \hyperref[C1]{(C1)} and \hyperref[C2]{(C2)} are equivalent, and both are implied by condition \hyperref[C3]{(C3)}.
Moreover, in those three cases, there exists an invariant probability measure $\pi$ for $(\P_t)_{t \geq 0}$ on $E$ and for all $x \in E$, 
\[ \lim \limits_{t \to \infty} \phi (\Hpi(t)) \|\P_t(x,\cdot) - \pi(\cdot) \|_{TV} = 0. \]
\end{thm}

The fact that (\ref{ConditionLyapSubGeom}) implies the convergence was proved by Douc, Fort and Guillin \cite{DFG}, see also Fort-Roberts \cite{FortRoberts} for the polynomial case, and was simplified, with stronger hypothesis and for the case of the total variation distance considered here, by Hairer \cite{Hairer}. The papers  \cite{FortRoberts} and \cite{DFG} also identify a moment condition similar to (\ref{ConditionMomentMainThm}), see Remark \ref{rmk:occupation_time_FR}, however they do not provide an equivalence result between the two conditions. 

\subsection{Links with previous results}

We make several remarks on this theorem and its connections to some known results.

\begin{rmk}[Converse implications and weak Poincaré inequalities]
\bb
An open natural question rose  by Theorem \ref{MainThm} is the following converse implication: starting from \hyperref[C2]{(C2)}, can we prove \hyperref[C3]{(C3)} ? We do not provide a firm answer to this seemingly difficult question, although the regeneration model studied in Section \ref{sec:examples} may be a good starting point for such a study. We also point out that it might be the case, for specific models, that \hyperref[C3]{(C3)} does not hold, but a weaker version of it does, in the sense that the function $\phi$ appearing in \hyperref[C3]{(C3)} might differ from the one in \hyperref[C2]{(C2)}. Such a weak convergence result would still be of interest, especially for diffusion operators, for which the existence of a functional $V$ satisfying \eqref{eq:Old_Lyap} is used in a paper by Cattiaux-Gozlan-Guillin-Roberto\cite{Cattiaux_GGR_2010} to derive some weighted Poincaré and weighted Cheeger inequalities. Using directly \hyperref[C2]{(C2)} to derive such weighted inequalities might also be possible, and we plan to tackle this apparently intricate question in future work. 
\end{rmk}

\begin{rmk} \bb 
Another possible extension is the inclusion of our new conditions \hyperref[C1]{(C1)} and \hyperref[C2]{(C2)} into the relations for the subgeometric case developed by Cattiaux and Guillin \cite[Theorem 1.2]{Cattiaux_Guillin_2017}. In the context of \cite{Cattiaux_Guillin_2017}, the state space is an unbounded set $D \subset \RR^d$, $d \ge 1$, whose unit outward normal vector at the boundary is denoted $n$ (if $\partial D$ is not void). The generator $\L$ is a differential operator which can be seen as the infinitesimal generator of a diffusion process.
Let us consider the condition (HLS1) from \cite[Corollary 3.3]{Cattiaux_Guillin_2017}.

\begin{enumerate}
\item [(HSL1)] There exists a smooth function $W : D \to \RR$ with $W \ge w > 0$, and some constants $\lambda, b > 0$ such that $\tfrac{\partial W}{\partial n} = 0$ on $\partial D$ and, for all $x \in D$,
\[ \L W(x) \le - \lambda |x|^2 W(x) + b \mathbf{1}_A(x), \]
for some bounded subset $A \subset D$.
\end{enumerate}

We show that (HSL1) implies (C2). Indeed, without loss of generality we can assume that $w = 1$. Consider $\phi(x) = \sqrt{x}$ which satisfies the hypothesis from Theorem \ref{MainThm}. A computation similar to the one performed in Subsection \ref{SubSec3imp2} shows that, setting for all $x \in D$, $t \ge 0$,
\[ \psi(t,x) := 2 \Hpi(H_{\phi}(W(x)) + t) - \Hpi(t), \]
we have
\begin{align*}
(\partial_t + &\L) \psi(t,x) \\
&\le 2 \phi \Big(\Hpi(H_{\phi}(W(x)) + t) \Big) + 2 \frac{\phi \Big(\Hpi(H_{\phi}(W(x)) + t) \Big)}{\phi(W(x))} \L W(x) - \phi(\Hpi(t)) \\
&\le 2 \phi \Big(\Hpi(H_{\phi}(W(x)) + t) \Big) + 2 \frac{\phi \Big(\Hpi(H_{\phi}(W(x)) + t) \Big)}{\phi(W(x))} (- \lambda |x|^2 W(x) + b \mathbf{1}_A(x)) \\
&\hspace{10cm} - \phi(\Hpi(t)).
\end{align*}
Noting that
\begin{align*}
\lambda |x|^2 \frac{W(x)}{\phi(W(x))} =  \lambda |x|^2 \sqrt{W(x)} \ge \lambda |x|^2 \mathbf{1}_{\{|x| \ge \tfrac{1}{\sqrt{\lambda}}\}}  \ge 1 - \mathbf{1}_{\big\{|x| \le \tfrac{1}{2 \sqrt{\lambda}}\big\}}, 
\end{align*}
where we used that $W \ge 1$ on $D$, we deduce that for some compact set $C$ with $A \cup \{|x| \le \tfrac{1}{2 \sqrt{\lambda}}\} \subset C$, for some constant $b_3 > 0$,  
\begin{align*}
(\partial_t + &\L) \psi(t,x) \le b_3 \phi \Big(\Hpi(H_{\phi}(W(x)) + t) \Big) \mathbf{1}_C(x) - \phi(\Hpi(t)).
\end{align*}
We then use the properties of $\Hpi$ and the fact that $W$ is smooth, hence bounded on $C$ compact, to conclude. Note that $C$ is also petite as a compact set of $\RR^d$, using that for diffusion processes such as those considered in \cite{Cattiaux_Guillin_2017}, compact sets of $\RR^d$ are petite, see \cite{Meyn_Stability_1}. The other requirements of \hyperref[C2]{(C2)} are easily seen to be satisfied. 

This provides valuable insights for the diffusion case of \cite{Cattiaux_Guillin_2017}. Indeed, since, according to their Theorem 1.2, some hitting time condition, (HLS2), that we do not detail here for the sake of conciseness, is equivalent to (HLS1), this shows that in this case, both notions of time (hitting time and occupation time) provide the same results. Additionally, this gives some converse implication: some log-Sobolev inequality, which is known to imply (HLS1), thus implies our condition (C2). We refer to \cite{Cattiaux_Guillin_2017} for more details on the relations that can be developed from (HLS1). The derivation of those other conditions starting from (C2) or (C1), or the converse implications, seem significantly more involved.   
\end{rmk}

\bb
\begin{rmk}
\label{rmk:occupation_time_FR}
In \cite[Theorem 1]{FortRoberts}, the occupation time condition is given, for the total variation distance, by the control of the quantity 
\[ \sup_{x \in C} \EE_x \Big[ \int_0^{\tau_C(\delta)} \phi \circ \Hpi (s) ds \Big], \]
for some $\delta > 0$. Since $\phi \circ \Hpi = (\Hpi)'$, this condition is very similar to \hyperref[C1]{(C1)}. Our condition is stronger, since it is equivalent to \hyperref[C2]{(C2)} from the above theorem and since we prove that \hyperref[C2]{(C2)} implies the condition of \cite{FortRoberts} to derive the convergence result. Yet the occupation time of \hyperref[C1]{(C1)} should be thought of as a regularized version of $\tau_C(\delta)$. In the exponential case, both random times give the same result regarding the asymptotic behavior, see Down-Meyn-Tweedie \cite[Lemma 4.2 and Theorem 6.5]{DownMT}, but the adaptation of the corresponding arguments to the subgeometric case is unclear. 
\end{rmk}

\begin{rmk}[The discrete-time case]
\bb
The condition \hyperref[C2]{(C2)} can be seen as an extension of a similar condition for the discrete-time case. Indeed, Tuominen and Tweedie give the following result:

\begin{thm}[Tuominen-Tweedie \cite{Tuominen_1994}]
\label{thm:Tuominen}
Assume that $P$ is the transition matrix, irreducible and aperiodic, of a Markov chain $(X_n)_{n \ge 0}$. Let $r$ a subgeometric rate function, i.e. $r$ is positive with $r(0) \ge 1$ and for some $r_0$ such that $r_0(1) \ge 2$, $r_0$ non-decreasing with $\frac{\ln(r_0(n))}{n} \overset{n \to \infty}{\to} 0$,
\[ 0 < \liminf_{n \to \infty} \frac{r(n)}{r_0(n)} \leq \limsup_{n \to \infty} \frac{r(n)}{r_0(n)} < \infty. \]
Suppose that there exist a sequence of extended real valued functions $(V_n)_{n \ge 0}$ with $V_n: E \to [1,\infty)$, a petite set $C \in \mathcal{B}(E)$ and $b < \infty$ constant such that $V_0$ is bounded on $C$, 
\[ \Big(V_0(x) = + \infty \Big) \implies \Big( V_1(x) = + \infty \Big), \]
and 
\[ P V_{n+1} - V_n \le -r(n) + b r(n) \mathbf{1}_C. \]
Then for all $x$ such that $ \EE_x[\sum_{k=0}^{\tau_C-1} r(k)] < \infty$, where $\tau_C := \inf \{ n > 0, X_n \in C\}$,
\begin{align*}
\lim \limits_{n \to \infty} r(n) \| P^n(x,\cdot) - \pi(\cdot)\|_{TV} = 0.
\end{align*}
\end{thm}
For $r(n) = \phi \circ \Hpi(n)$ which is a subgeometric rate (see Lemma \ref{lemma:subgeometric}) and $V_n(x) = \psi(n,x)$, we see that condition \hyperref[C2]{(C2)} from Theorem \ref{MainThm} corresponds to the discrete-time condition of Theorem \ref{thm:Tuominen}. One might view \hyperref[C2]{(C2)} as a continuous counterpart of the condition from Theorem \ref{thm:Tuominen}. As in our case, the discrete-time counterpart to \hyperref[C3]{(C3)} implies the condition of Theorem \ref{thm:Tuominen}, see Douc-Fort-Moulines-Soulier \cite[Proposition 2.1]{DFMS_2004}. 

\end{rmk}
\dd

The remaining part of this note is organized as follows. In Section \ref{Setting}, we recall the main definitions of the theory of convergence for continuous-time strong Markov processes, and define our notion of extended generator, following \cite{Davis18}. 
In Section \ref{SectionEquivalence}, we prove the new results of Theorem \ref{MainThm} above. \bb In Section \ref{sec:examples} we give two examples of application. The first one is a regeneration model appearing in collisionless kinetic theory. For this model, we prove that condition \hyperref[C1]{(C1)} of Theorem \ref{MainThm} holds, and provide several results pointing towards the non-existence of a functional $V$ such that condition \hyperref[C3]{(C3)} holds (although a complete proof still lacks at the moment). The second example is the study of the Compound Poisson-process driven Ornstein-Uhlenbeck process, for which it is known from \cite{FortRoberts, DFG} that \hyperref[C3]{(C3)} applies and for which we find an explicit function $\psi$ such that \hyperref[C2]{(C2)} holds. 
\dd 

\section{Setting, definitions and preliminary results}
\label{Setting}

\subsection{Setting and definitions}
Let $X = (X_t)_{t \geq 0}$ be a continuous-time strong Markov process with values in a Polish space $E$. For $x \in E$, we write $\PP_x$ for the probability measure such that $\PP_x(X_0 = x) = 1$, $\EE_x$ the corresponding expectation. We denote by $(\P_t)_{t \geq 0}$ the corresponding semigroup: for all functions  $f$ in $ \mathcal{B}_b(E)$ with $\mathcal{B}_b(E) = \{f: E \to \RR, f \text{ measurable and bounded} \},$  for all $x \in E$, we have $\P_tf(x) = \EE_x[f(X_t)]$. We set, for $f \in \mathcal{B}_b(E)$, $x \in E$, $\hat{\mathcal{L}}f(x) = \frac{d}{dt} \EE_x[f(X_t)]|_{t = 0}$ provided this object
exists. We call $\hat{\L}$ the (strong) generator  and $\mathcal{D}(\hat{\mathcal{L}})$ its domain given by $$\mathcal{D}(\hat{\mathcal{L}}) = \Big\{f: E \to \RR, \forall x \in E, \lim \limits_{t \to 0} \frac{\P_t f(x) - f(x)}{t} \text{ exists} \Big\} . $$

Let us recall some more definitions. We say that a continuous-time Markov process $(X_t)_{t \geq 0}$ with values in $E$ is non-explosive if there exists a family of pre-compact open sets $(O_n)_{n \geq 0}$ such that $O_n \to E$ as $n \to \infty$, and such that, setting for all $m \geq 0$, $T_m = \inf \{t > 0, X_t \not \in O_m\}$, for all $x \in E$, $$\PP_x \Big(\lim_{m \to \infty} T_m = \infty \Big) = 1.$$

We say that $(X_t)_{t \geq 0}$ is $\varphi$-irreducible for some $\sigma$-finite measure $\varphi$ if $\varphi(B) > 0$ implies that for all $B \in \mathcal{B}(E)$, for all $x\in E$, $\EE_x[\int_0^{\infty} \mathbf{1}_{B}(X_s) ds] > 0$. A $\varphi$-irreducible process admits a maximal irreducibility measure $\psi$ such that $\mu$ is absolutely continuous with respect to $\psi$ for any other irreducibility measure $\mu$ \cite{Nummelin84}. A set $A \in \mathcal{B}(E)$ such that $\psi(A) > 0$ for some maximal irreducibility measure $\psi$ is then said to be accessible, and full is $\psi(A^c) = 0$. A set $A \in \mathcal{B}(E)$ is said to be absorbing if $\PP_x(X_t \in A) = 1$ for all $x \in A$, $t \geq 0$.
We simply say that $(X_t)_{t \geq 0}$ is irreducible if it is
$\varphi$-irreducible for some $\sigma$-finite measure $\varphi$.

A non-empty measurable set $C$ is said to be petite if there exists a probability measure $a$ on
$\mathcal{B}(\RR_+)$ and a non-trivial $\sigma$-finite measure $\nu$ on $\mathcal{B}(E)$ such that
$$ \forall x \in C, \int_0^{\infty} \P_t(x,\cdot) a(dt) \geq \nu(\cdot). $$

We say that a process $(X_t)_{t \geq 0}$ with associated semigroup $(\P_t)_{t \geq 0}$ is aperiodic if there exists an $m > 0$ such that, denoting by $\delta_m$ the Dirac mass at $m$, there exists an accessible $\delta_m$-petite set $C$ (i.e. petite with measure $a=\delta_m$ on $\RR_+$)
and some $t_0 \geq 0$ such that for all $x \in C$, $t \geq t_0$, $\P_t(x,C) > 0$.

We assume furthermore that our process is Feller, in the sense that for all $t>0$, all
continuous bounded function $f : E \to \R$, the function $\P_tf:E \to \R$ is also continuous.


The (weak) Feller property implies that $(X_s)_{s \geq 0}$ has a c\`adl\`ag modification, which
we will always consider
from now on,  see for instance \cite[Theorem 2.7]{Revuz91}. In particular,
the hitting times of closed sets are stopping times.
%
%

We have the following result on $\mathcal{D}(\hat{\mathcal{L}})$.

\begin{prop}\cite[Propositions 14.10 and 14.13]{Davis18}
For $f \in \mathcal{D}(\hat{\mathcal{L}})$, for all $x\in E$, all $t \geq 0$, we have
$\int_0^t |\hat{\mathcal{L}}f(X_s)|ds < \infty$ $\PP_x$-a.s.
Moreover, defining the real-valued process $(C_t^f)_{t \geq 0}$ by
$$ C_t^f = f(X_t) - f(X_0) - \int_0^t \hat{\mathcal{L}}f(X_s) ds, $$
the process $(C_t^f)_{t \geq 0}$ is a  $\PP_x$-local martingale for any $x\in E$.
\end{prop}

Following Davis \cite{Davis18}, we define an extension of the generator $\hat{\mathcal{L}}$ in the following way.

\begin{defi}
Let $\mathcal{D}(\L)$ denote the set of measurable functions $f: E \to \mathbb{R}$ with the following property: there exists a measurable function $h: E \to \RR$ such that for all $x \in E$,
there holds $\PP_x(\forall t \geq 0, \int_0^t |h(X_s)| ds < \infty ) = 1$, and the process 
$$ C_t^f = f(X_t) - f(X_0) - \int_0^t h(X_s) ds, $$
is a $\PP_x$-local martingale. In this case, we set $\mathcal{L}f := h$. We
call $(\mathcal{L}, \mathcal{D}(\mathcal{L}))$ the extended generator of $(X_t)_{t \geq 0}$. 
\end{defi}

The extended generator is indeed an extension: we have $\mathcal{D}(\hat{\mathcal{L}}) \subset \mathcal{D}(\mathcal{L})$ and $\mathcal{L}$ and $\hat{\mathcal{L}}$ coincide on $\mathcal{D}(\hat{\mathcal{L}})$. Following \cite{Davis18} again, we introduce the following notation.  

\begin{notat}
\label{NotatLf}
For $f:E \to \RR$, for $g: E \to \RR$ measurable such that $\int_0^t |g(X_s)| ds < \infty$ for all $t \geq 0$, $\PP_x$-almost surely for all $x \in E$, we write
$$ \mathcal{L} f \leq g $$
if the process 
$$ f(X_t) - f(x) - \int_0^t g(X_s) ds $$
is a $\PP_x$-local supermartingale for all $x \in E$.
\end{notat}

\begin{rmk}[\cite{Hairer}]
It is possible to have $\L f \leq g$ even in situations where $f$ does not belong to the
extended domain of $\L$. For instance, take $f(x) = - |x|$ when $(X_t)_{t \geq 0}$ is a Brownian motion. In this case, one has $\L f \leq 0$, but $f \not \in \mathcal{D}(\L)$, and \textit{a fortiori}  $f \not \in \mathcal{D}({\hat{\L}})$.
\end{rmk}

Similarly, we introduce

\begin{notat}
\label{NotatLf2}
If $j: \RR_+ \times E \to \RR$ is $C^1$ in its first argument, for $k: \RR_+ \times E \to \RR$ measurable such that
for all $t \geq 0$, we have $\int_0^t |k(s,X_s)| ds < \infty$ $\PP_x$-a.s. for all $x \in E$,
we write
$$ (\partial_t + \L)j \leq k $$
\[ \hbox{if} \qquad  M_t := j(t,X_t) - j(0,x) - \int_0^t k(s,X_s) ds \]
is a $\PP_x$-local supermartingale for all $x \in E$.
\end{notat}

In this note, we use the following definition of the total variation distance: for two probability measures $\mu$, $\nu$ on $E$, we set 
$$ \|\mu - \nu\|_{TV} = \frac{1}{2} \sup_{A \in \mathcal{B}(E)}|\mu(A) - \nu(A)|. $$
As a consequence, we have 
$$ \|\mu - \nu\|_{TV} = \inf_{Z \sim \mu, Z' \sim \nu} \PP(Z \ne Z'), $$
where the infimum is taken over all couples of random variables  such that $Z$ has law $\mu$ and $Z'$ has law $\nu$.

\subsection{Extended generator and local martingales}

As we are working in an abstract framework, we heavily use the extended generator (see Notations \ref{NotatLf} and \ref{NotatLf2}) and the inequalities of the form $$\L f \leq g \qquad \text{ and } \qquad (\partial_t + \L) \psi \leq \psi_2. $$
For this reason, we will use several preliminary results from \cite{Hairer} that we detail below.

\begin{prop} 
\label{PropHairerMartingale} 
Let $(y_t)_{t \geq 0}$ be a real-valued c\`adl\`ag semimartingale and let $ \varphi: \RR_+ \times \RR \to \RR$ be a function that is $C^1$ in its first argument, and $C^2$ and concave in its second argument. Then, the process 
$$ \varphi(t,y_t) - \int_0^t \partial_x \varphi(s,y_{s-})dy_s - \int_0^t \partial_t \varphi(s, y_{s-}) ds$$
is non-increasing.
\end{prop}

\begin{proof} As $(y_t)_{t \geq 0}$ is a semimartingale, we can write it as $y_t = A_t + M_t$, where $(A_t)_{t \geq 0}$ is a process of finite variation and $(M_t)_{t \geq 0}$ is a local martingale.  From It\^o's formula for c\`adl\`ag processes, see for instance \cite[Theorem 4.57]{Jacod87}, we then have
\begin{align*}
 \varphi(t,y_t) &= \varphi(0,y_0) + \int_0^t \partial_x \varphi(s,y_{s-}) dy_s + \int_0^t \partial_t \varphi(s,y_{s-}) ds \\
 &\quad + \int_0^t \partial^2_x \varphi(s,y_{s-}) d \langle M \rangle^c_s + \sum_{s \in [0,t]} \Big(\varphi(s,y_s) - \varphi(s,y_{s-}) - \partial_x \varphi(s,y_{s-}) \Delta y_s \Big), 
 \end{align*}
where $\langle M \rangle^c_t$ denotes the quadratic variation of the continuous part of $M$ at time $t$ and with $\Delta y_s$ defined by
$\Delta y_s = y_{s} - y_{s-}.$ Since $\langle M \rangle^c_t$ is an increasing process, and $\partial^2_x \varphi(\cdot, \cdot) \leq 0$ by hypothesis, the claim follows. 
\end{proof}

Recall that we write $\L$ for the extended generator of our $E$-valued
Markov process $(X_t)_{t \geq 0}$.

\begin{coroll}
\label{CorollHairer}
Let $F, G : E \to \RR$ such that 
$$ \L F \leq G$$
in the sense of Notation \ref{NotatLf}.
Then, if $ \varphi: \RR_+ \times \RR \to \RR$ is a function that is $C^1$ in its first argument, and $C^2$ and concave in its second argument with additionally $\partial_x \varphi \geq 0$, then for all $t \geq 0$, all $x \in E$,
$$ (\partial_t + \L ) \varphi(t,F(x)) \leq \partial_t \varphi(t,F(x)) + \partial_x \varphi(t,F(x)) G(x), $$
in the sense of Notation \ref{NotatLf2}. 
\end{coroll}

\begin{proof}
Set $y_t = F(X_t)$ for all $t \geq 0$. We have
$$ d y_t = G(X_t) dt + dN_t +dM_t, $$
with $M$ a c\`adl\`ag  local martingale such that $M_0 = 0$ and $N$ a non-increasing process.
By Proposition \ref{PropHairerMartingale}, there is a non-increasing process $(R_t)_{t \geq 0}$
such that
$$
d \varphi(t,y_t) = \partial_x \varphi(t,y_{t-}) dy_t + \partial_t \varphi(t,y_{t-}) dt + dR_t,
$$
so that
$$d \varphi(t,y_t) = \partial_x \varphi(t,y_{t-}) (G(X_t)dt + dN_t + dM_t) + \partial_t \varphi(t,y_{t-}) dt + dR_t.
$$
Since $\partial_x \varphi$ is non-negative, the process
$$  \varphi(t,y_t) - \varphi(0,y_0) - \int_0^t \Big( \partial_t \varphi(s,y_{s-}) + \partial_x \varphi(s,y_{s-}) G(X_s) \Big) ds $$ is indeed a local supermartingale
(as sum of a local martingale and of a non-increasing process).
\end{proof}

\subsection{Properties of $\phi$ and $\Hpi$}

 We recall that  $\phi:[1,\infty)\to [1,\infty)$ is $C^1$, 
strictly increasing, strictly concave
such that $\phi(x)\leq x$ for all $x\geq 1$,
$\frac{\phi(x)}{x} \downarrow 0$ and 
$\phi(x) - x \phi'(x) \uparrow \infty$ when $x \to \infty$.
The function $H_{\phi}$ is defined, for all $u \geq 1$ by
$$H_{\phi}(u) = \int_1^u \frac{ds}{\phi(s)}, $$
and we consider the corresponding inverse function $H_{\phi}^{-1}:[0,\infty)\to[1,\infty)$.
\begin{lemma} The following inequality holds:
\begin{equation}\label{huu}
\Hpi(s+t)\leq\Hpi(s)\Hpi(t) \quad \hbox{for all $s,t\geq0$.}
\end{equation}
\end{lemma}
\begin{proof}
Set $g(\cdot) := (\ln \circ \Hpi)(\cdot)$, and consider the function given, for all $s, t \geq 0$, by
\[h(s,t) := g(s+t) - g(s) - g(t). \]
For all $s \geq 0$, $h(s,0) = 0$ since $\Hpi(0) = 1$. Moreover, using that $(\Hpi)'(u) = (\phi \circ \Hpi)(u)$ for all $u \geq 0$,
\[\partial_t h(s,t) = \frac{\phi(\Hpi(t+s))}{\Hpi(t+s)} - \frac{\phi(\Hpi(t))}{\Hpi(t)} \leq 0, \]
using that $\frac{\phi(x)}{x} \downarrow 0$ as $x \to \infty$. Hence $h(s,t) \leq 0$ for all $s,t \geq 0$ and the conclusion follows by taking the exponential.
\end{proof}
An immediate study also shows that    
\begin{equation}\label{hhu}
\phi(\kappa x) \leq \kappa \phi(x) \quad \hbox{for all $x\geq 1$, all $\kappa\geq 1$}.
\end{equation}


\begin{lemma}
\label{lemma:subgeometric}
\bb 
The function $\phi \circ \Hpi$ is a rate function in the sense of Fort-Robert \cite{FortRoberts}, i.e. it is bounded on bounded intervals, non-decreasing with
\[  \frac{\ln(\phi \circ \Hpi(s))}{s} \downarrow 0, \]
as $s \to \infty$. In particular,
\begin{align}
\label{eq:ineqPhiCircH}
\phi \circ \Hpi(t+s) \le 2 \Big(\phi \circ \Hpi(s) \Big) \Big( \phi \circ \Hpi(t) \Big), \qquad \forall s, t \ge 0. 
\end{align}
\end{lemma}

\begin{proof}
\bb 
Set, for all $s \ge 0$, $g(s) := \ln(\Hpi(s))$, $h(s) := \frac{\ln(\phi \circ \Hpi(s))}{s}$. 

\begin{enumerate}
\item \textbf{Proof that $\lim \limits_{s \to \infty} h(s) = 0$.} We have $g'(s) = \frac{\phi \circ \Hpi(s)}{\Hpi(s)} \to 0$ as $s \to \infty$ by hypothesis on $\phi$. Therefore $\frac{g(s)}{s} \to 0$ as $s \to \infty$ (using L'Hospital rule), and since $\Hpi(s) \ge 1$ for all $s \ge 0$, since $\phi(x) \le x$ for all $x \ge 1$ and using that $\phi$ increases to infinity,
\[ 0 \le h(s) \le \frac{g(s)}{s}, \]
and we conclude that $h(s) \to 0$ as $s \to \infty$.
\item \textbf{Proof that $h$ is non-increasing outside a compact set.}
We first have
\begin{align*}
h'(s) = \frac{\phi'(\Hpi(s))}{s} - \frac{h(s)}{s}.
\end{align*} 
We have, for all $s > 0$ (so that $\Hpi(s) > 1$ and $\phi \circ \Hpi(s) > 1$),
\begin{align*}
\lim \limits_{s \to \infty} \frac{\phi'(\Hpi(s))}{h(s)} = \lim \limits_{s \to \infty} \frac{(s \phi'(\Hpi(s)))'}{(\ln(\phi \circ \Hpi(s)))'}
\end{align*}
using again l'Hospital rule, and 
\begin{align*}
\frac{(s \phi'(\Hpi(s)))'}{(\ln(\phi \circ \Hpi(s)))'} = 1 + \frac{s \phi''(\Hpi(s)) \phi \circ \Hpi(s)}{\phi'(\Hpi(s))} \le 1,
\end{align*}
since $\phi$ is strictly concave and increasing. Hence $h$ is non-increasing for $s$ large enough.
\item \textbf{Proof of \eqref{eq:ineqPhiCircH}.} We use a result from Thorisson \cite[Lemma 1]{Thorisson_Queue_1985} which applies to subgeometric rate functions greater than 2 with the choice $s \to 2\phi \circ \Hpi(s)$. \qedhere
\end{enumerate}
\end{proof}

We will use several times the following remark, based on the definition of $\tilde{\tau}_C^r$, see (\ref{DefTildeTauC}).
\begin{rmk}
\label{RmkHpiTauC}
For all $x \in E$ and all non-decreasing $C^1$ function $f:\RR_+\to\RR_+$ such that $f(0)=0$,
\[
\EE_x[f(\tilde{\tau}_C^r)] = \EE_x \Big[ \int_0^{\infty} e^{-r \int_0^s \mathbf{1}_C(X_u) du} f'(s)
ds \Big].
\]
Indeed, it suffices to use that $\EE_x[f(\tilde{\tau}_C^r)]=\int_0^\infty
\PP_x(\tilde{\tau}_C^r \geq s) f'(s) ds$ and that
$$
\PP_x(\tilde{\tau}_C^r \geq s)=\PP_x\Big(T\geq r \int_0^s \mathbf{1}_C(X_u) du\Big)=
\EE_x\Big[ e^{-r \int_0^s \mathbf{1}_C(X_u) du}\Big].
$$
\end{rmk}

\section{Proof of Theorem \ref{MainThm}}
\label{SectionEquivalence}

In this section, we give the proofs of the results stated in Theorem \ref{MainThm}.

\subsection{Proof that \hyperref[C3]{(C3)} implies Condition \hyperref[C2]{(C2)}}
\label{SubSec3imp2}
We introduce $\psi_0:\RR_+\times [1,\infty) \to [1,\infty)$
defined by $\psi_0(t,x) = \Hpi(H_{\phi}(x) + t)$.  This function is $C^1$ in its first argument $t$ and $C^2$
in its second argument. Moreover, for all $t\geq 0$, all $x\geq 1$,
$$ \partial_x \psi_0(t,x) = H_{\phi}'(x) (\Hpi)'(H_{\phi}(x) + t)  = \frac{\phi \Big( \Hpi \big( H_{\phi}(x) + t\big) \Big)}{\phi(x)} \geq 0. $$
Next, 
\begin{align*}
\partial_x^2 \psi_0(t,x) &= \frac{\phi' \Big( \Hpi \big( H_{\phi}(x) + t\big) \Big) \phi \Big( \Hpi \big( H_{\phi}(x) + t\big) \Big) - \phi'(x) \phi \Big( \Hpi \big( H_{\phi}(x) + t\big) \Big) }{\phi^2(x)}  \\
&= \frac{\phi \Big( \Hpi \big( H_{\phi}(x) + t\big) \Big)}{\phi^2(x)} \Big(  \phi' \Big( \Hpi \big( H_{\phi}(x) + t\big) \Big) - \phi'(x) \Big)\leq 0,
\end{align*} 
since the first factor is positive, while the second one is negative because $\phi'$ is
decreasing and $x \leq \Hpi(H_{\phi}(x) + t)$.
We conclude that $\psi_0$ satisfies the assumption of Corollary \ref{CorollHairer}.
We set $\psi(t,x) = 2\psi_0(t,V(x)) - \Hpi(t)$. \bb On the one hand, using $\phi(x) \le x$ for all $x \ge 1$ and \eqref{huu},
\[ 
\Hpi(t) = 2 \Hpi(t) - \Hpi(t) \leq 2\psi_0(t,V(x)) - \Hpi(t) = \psi(t,x) \le \psi(0,x) \Hpi(t)
\] 
\dd  for all $t \geq 0$, all
$x \in E$, and,  using Corollary \ref{CorollHairer} and that $(\Hpi)' = \phi \circ \Hpi$, one has
\begin{align*}
(\partial_t& + \L) \psi(t,x) \leq 2 \partial_t \psi_0(t,V(x)) + 2 \partial_x \psi_0(t,V(x)) \L V(x) - \phi (\Hpi(t)) \\
&= 2 \phi \Big( \Hpi \big(H_{\phi}(V(x)) + t\big) \Big) + 2\frac{\phi \Big( \Hpi \big(H_{\phi}(V(x)) + t\big) \Big)}{\phi(V(x))} \L V(x) - \phi( \Hpi(t)) \\
&\leq 2\phi \Big( \Hpi \big(H_{\phi}(V(x)) + t\big) \Big) + 2\frac{\phi \Big( \Hpi \big(H_{\phi}(V(x)) + t\big) \Big)}{\phi(V(x))} (- \phi(V(x)) + K \mathbf{1}_C(x) ) - \phi( \Hpi(t)) \\
&\leq 2K \frac{\phi \Big( \Hpi \big(H_\phi(V(x)) + t\big) \Big)}{\phi(V(x))} \mathbf{1}_C(x) - \phi ( \Hpi(t)),
\end{align*} 
where we used the bound on $\L V$ from condition \hyperref[C3]{(C3)}. \bb Recall that, by hypothesis, $C$ is compact and $V$ is continuous. Since $\phi$ is also continuous with $\phi(x) \le x$ for all $x \ge 1$, we have, for some $K_2 > 1$
\[
\sup_{x \in C} \frac{V(x)}{\phi(V(x))} \le K_2.
\]
Using now \eqref{huu} and \eqref{hhu} (recall that $\Hpi(t)\geq 1$), we conclude that 
$$ (\partial_t + \L) \psi(t,x) \leq 2K \frac{\phi \Big(\Hpi(t)V(x)\Big)}{\phi(V(x))} \mathbf{1}_C(x) - \phi( \Hpi(t)) \leq 2KK_2 \phi(\Hpi(t)) \mathbf{1}_C(x) - \phi( \Hpi(t)). $$
\dd 
We also have $\psi(0,x) = 2V(x) - 1$, so that indeed $\sup_{x \in C}  \psi(0,x)<\infty$
(because $C$ is compact and $V$ has precompact sublevel sets), and, using that $\L 1 = 0$, recalling \hyperref[C3]{(C3)}, that $V\geq 1$ and that $\phi$ is non-decreasing,
\[ \L \psi(0,x) = 2K \mathbf{1}_C(x) - 2 \phi(V(x)) \leq 2K \mathbf{1}_C(x) - 2 \phi(1), \]
which completes the proof.

\subsection{Proof that \hyperref[C2]{(C2)} implies \hyperref[C1]{(C1)}}

 Let $x \in E$ and set, for all $t \geq 0$, 
\begin{align*}
M_t = \psi(t,X_t) - \psi(0,x) - \bb \kappa \int_0^t \mathbf{1}_C(X_s) \phi(\Hpi(s)) ds \dd 
+ \int_0^t \phi (\Hpi(s)) ds,
\end{align*}
\bb then by \hyperref[C2]{(C2)}, $(M_t)_{t \geq 0}$ is a $\PP_x$-local supermartingale starting at 0\dd. Hence there exists an increasing to infinity sequence $(\sigma_i)_{i \geq 1}$ of stopping times such that for all $i \geq 1$, $(M_{t \wedge \sigma_i})_{t \geq 0}$ is a bounded supermartingale. 

\vspace{.5cm}

\textbf{Step 1.} We introduce the stopping time 
\begin{align*}
\tilde{\tau}^1 = \inf \Big\{t \geq 0, \int_0^t \mathbf{1}_C(X_u) du \geq \frac{1}{2 \kappa} \Big\}, 
\end{align*}
and note that $X_{\tilde{\tau}^1} \in C$ almost surely. In this step, we show that for all $x \in E$,
\[ \EE_x[\Hpi(\tilde{\tau}^1)] \leq \bb 2 (\psi(0,x) + 1). \] 
\bb
For all $i \geq 1$, using that $\psi \ge 0$ and the definition of $(M_t)_{t \ge 0}$,
\begin{align*}
0
&\leq \EE_x[\psi(\tilde{\tau}^1 \wedge \sigma_i,X_{\tilde{\tau}^1 \wedge \sigma_i})]\\
&= \EE_x \Big[\psi(0,x) + \kappa \int_0^{\tilde{\tau}^1 \wedge \sigma_i} \mathbf{1}_C(X_u)
  \phi(\Hpi(u)) du - \int_0^{\tilde{\tau}^1 \wedge \sigma_i} \phi(\Hpi(s)) ds + M_{\tilde{\tau}^1 \wedge \sigma_i} \Big].
  \end{align*}
  Hence, using that $\tilde{\tau}^1 \wedge \sigma_i \leq \tilde{\tau}^1$, that $\phi \circ \Hpi$ is non-decreasing, that $\phi \circ \Hpi(t) \le \Hpi(t)$ for all $t \ge 0$ and the definition of $\tilde{\tau}^1$,
  \begin{align*}
  \EE_x \Big[\int_0^{\tilde{\tau}^1 \wedge \sigma_i} \phi(\Hpi(u)) du \Big]
&\leq  \psi(0,x) + \kappa \EE_x \Big[ \int_0^{\tilde{\tau}^1 \wedge \sigma_i} \mathbf{1}_C(X_u)
  \phi(\Hpi(u)) du\Big]\\
&\leq \psi(0,x) + \kappa \EE_x \Big[ \phi(\Hpi(\tilde{\tau}^1 \wedge \sigma_i)) \int_0^{\tilde{\tau}^1} \mathbf{1}_C(X_u) du \Big] \\
&\le \psi(0,x) + \frac{1}{2} \EE_x[\Hpi(\tilde{\tau}^1 \wedge \sigma_i)].
\end{align*}
Since $(\Hpi)' = \phi \circ \Hpi$ and $\Hpi(0) = 1$, we obtain that for all $i \geq 1$,
\[ \EE_x[\Hpi(\tilde{\tau}^1 \wedge \sigma_i)] \leq 2 (\psi(0,x) + 1), \]
and an application of the monotone convergence theorem allows us to conclude.

\vspace{.5cm}

\dd

\textbf{Step 2.} We consider the quantity defined for all $x \in E$, for $\rho \geq 0$ and $r > 0$ by
\[A_{x,\rho,r} := \EE_x \Big[ \int_0^{\infty} e^{-r \int_0^s \mathbf{1}_C(X_u) du} (\Hpi)'(s) e^{-\rho s^2} ds \Big] \]
which is finite because $(\Hpi)'(s)=\phi(\Hpi(s))\leq \Hpi(s)$, whence
\[ \phi(\Hpi(s)) \le \Hpi(s) \leq \Hpi(0)e^s=e^s. \]
 We have
\begin{align*}
A_{x,\rho,r} &= \EE_x \Big[ \int_0^{\tilde{\tau}^1} e^{-r \int_0^s \mathbf{1}_C(X_u) du} (\Hpi)'(s) e^{-\rho s^2} ds \Big]   + \EE_x \Big[ \int_{\tilde{\tau}^1}^{\infty} e^{-r \int_0^s \mathbf{1}_C(X_u) du} (\Hpi)'(s) e^{-\rho s^2} ds \Big] \\
&\leq \EE_x \Big[ \int_0^{\tilde{\tau}^1} (\Hpi)'(s) ds \Big] + \EE_x \Big[ \int_{\tilde{\tau}^1}^{\infty} e^{-r \int_0^{\tilde{\tau}^1} \mathbf{1}_C(X_u) du} e^{-r \int_{\tilde{\tau}^1}^s \mathbf{1}_C(X_u) du} (\Hpi)'(s) e^{-\rho s^2} ds \Big]
\\
&\le \EE_x[\Hpi(\tilde{\tau}^1)] + \EE_x \Big[ e^{-r \int_0^{\tilde{\tau}^1} \mathbf{1}_C(X_u) du}
  \int_{\tilde{\tau}^1}^{\infty} e^{-r \int_{\tilde{\tau}^1}^s \mathbf{1}_C(X_u) du} (\Hpi)'(s) e^{- \rho s^2} ds \Big].
\end{align*}
Using the strong Markov property
\begin{align*}
A_{x,\rho,r} &\le \EE_x[\Hpi(\tilde{\tau}^1)]  \\
&\qquad + \EE_x \Big[ e^{-r \int_0^{\tilde{\tau}^1} \mathbf{1}_C(X_u) du} \EE_{X_{\tilde{\tau}^1}} \Big[ \int_0^{\infty} e^{-r \int_0^s \mathbf{1}_C(X_u) du} (\Hpi)'(\tilde{\tau}^1 + s) e^{- \rho (s + \tilde{\tau}^1)^2} ds \Big] \Big] \\
&\le \EE_x[\Hpi(\tilde{\tau}^1)] + \EE_x \Big[ e^{-r \int_0^{\tilde{\tau}^1} \mathbf{1}_C(X_u) du} \Hpi(\tilde{\tau}^1) \EE_{X_{\tilde{\tau}^1}} \Big[ \int_0^{\infty} e^{-r \int_0^s \mathbf{1}_C(X_u) du} (\Hpi)'( s) e^{- \rho s^2} ds \Big] \Big]
\end{align*}
because $(\Hpi)'(\tilde{\tau}^1 + s)=\phi(\Hpi(\tilde{\tau}^1 + s))\leq\phi(\Hpi(\tilde{\tau}^1)\Hpi(s))\leq \Hpi(\tilde{\tau}^1)\phi(\Hpi(s))$ by (\ref{huu}) and  (\ref{hhu}).
Using the definition of $A_{x,\rho,r}$ and the fact that $X_{\tilde{\tau}^1} \in C$, we conclude that 
\begin{align}
\label{IneqAxrho}
A_{x,\rho,r} &\leq \EE_x[\Hpi(\tilde{\tau}^1)] +\EE_x \Big[ e^{-r \int_0^{\tilde{\tau}^1} \mathbf{1}_C(X_u) du} \Hpi(\tilde{\tau}^1) \Big]\sup_{y \in C} A_{y,\rho,r}.
\end{align}

\vspace{.5cm}

\textbf{Step 3.} We now prove that there is $r_0>0$ (large) such that
\[ \sup_{x \in C} \EE_x \Big[ e^{-r_0 \int_0^{\tilde{\tau}^ 1} \mathbf{1}_C(X_u) du} \Hpi(\tilde{\tau}^1) \Big] \leq \frac{1}{2}. \]
By definition of $\tilde{\tau}^1$,
$\int_0^{\tilde{\tau}^1} \mathbf{1}_C(X_u) du= \frac{1}{2\kappa}$ a.s.
Hence, for all $x \in E$,
\bb
\begin{align*}
\EE_x \Big[ e^{-r_0 \int_0^{\tilde{\tau}^ 1} \mathbf{1}_C(X_u) du} \Hpi(\tilde{\tau}^1) \Big]
= \EE_x \Big[ e^{-\frac{r_0}{2\kappa}} \Hpi(\tilde{\tau}^1) \Big] \leq 2 e^{-\frac{r_0}{2\kappa}}(\psi(0,x) + 1)
\end{align*}
by Step 1. Since $0 < \sup_{x\in C} \psi(0,x) \le \kappa < \infty$ by assumption, the conclusion follows.

\vspace{.5cm}

\dd 
\textbf{Step 4.} Coming back to (\ref{IneqAxrho}), choosing $r=r_0$ given by Step 3 and taking the supremum
over $x \in C$ on both sides and using Step 3, we find
\[ \sup_{x \in C} A_{x,\rho,r_0} \leq \sup_{x \in C} \EE_x[\Hpi(\tilde{\tau}^1)] +
\frac{1}{2} \sup_{x \in C} A_{x,\rho,r_0}, \]
so that, using Step 1 and that $\psi(0,\cdot) \leq \kappa$ on $C$,
\begin{align}
\label{eq:bound_Axrho}
 \bb  \sup_{x \in C} A_{x, \rho,r_0} \leq 4 (\kappa + 1). 
 \end{align}
We now apply Fatou's lemma and Remark \ref{RmkHpiTauC},
\bb
\begin{align*}
\sup_{x \in C} \EE_x[\Hpi(\tilde{\tau}_C^{r_0})] &= \sup_{x \in C} \EE_x \Big[ \int_0^{\infty} e^{-r_0 \int_0^s \mathbf{1}_C(X_u) du} (\Hpi)'(s) ds \Big]
\leq \sup_{x \in C} \liminf \limits_{\rho \to 0} A_{x,\rho,r_0} \le 4 (\kappa+1).
\end{align*}

\vspace{.5cm}

\textbf{Conclusion}
We come back to (\ref{IneqAxrho}) using the results of Step 1 and Step 4. For all $x \in E$, using that $e^{-r_0 \int_0^{\tilde{\tau}^1} \mathbf{1}_C(X_u) d u} \le 1$ and \eqref{eq:bound_Axrho},
\begin{align*}
A_{x,\rho,r_0} &\leq \EE_x[\Hpi(\tilde{\tau}^1)] + \EE_x \Big[ e^{-r_0 \int_0^{\tilde{\tau}^1} \mathbf{1}_C(X_u) du} \Hpi(\tilde{\tau}^1) \Big] \sup_{x \in C} A_{x,\rho,r_0} \\
&\leq \EE_x[\Hpi(\tilde{\tau}^1)] (5 + 4\kappa) \\
&\leq 2 (\psi(0,x) +1)  (5 + 4 \kappa).
\end{align*}
Hence, as in Step 4,
\begin{align*}
  \EE_x[\Hpi(\tilde{\tau}_C^{r_0})] &= \EE_x \Big[ \int_0^{\infty} e^{-r_0 \int_0^s \mathbf{1}_C(X_u) du} (\Hpi)'(s) ds \Big]\leq \liminf \limits_{\rho \to 0} A_{x,\rho,r_0} \leq  2 (\psi(0,x) +1)  (5 + 4 \kappa),
\end{align*}
and the conclusion follows.
\dd 

\subsection{Proof that \hyperref[C1]{(C1)} implies \hyperref[C2]{(C2)}}

We fix $r>0$ so that \hyperref[C1]{(C1)} holds and recall that the randomized hitting time is given by
$$ \tilde{\tau}_C^r = \inf \Big\{t > 0, \int_0^t \mathbf{1}_C(X_s) ds > \frac{T}{r} \Big\}, $$
where $T$ is a random variable with exponential law of parameter 1 independent of
everything else. For the sake of simplicity we will omit the superscript $r$ in what follows
and write $\tilde{\tau}_C=\tilde{\tau}_C^r$. Our goal is to show that
$$ 
\psi(t,x)=\EE_x \Big[\Hpi(\tilde{\tau}_C + t) \Big] =
\EE_x \Big[\int_0^{\infty} e^{-r \int_0^s \mathbf{1}_C(X_u) du} (\Hpi)'(s + t) ds \Big]
$$
satisfies \hyperref[C2]{(C2)}. Here, the second equality follows from Remark \ref{RmkHpiTauC}.

\vspace{.5cm}

\bb We of course have, for all $t \ge 0$ and $x \in E$
\[ \bb 
\Hpi(t) \le \psi(t,x) \le \Hpi(t) \psi(0,x) \]
using \eqref{huu} and \eqref{hhu}. Moreover,
$\kappa=\sup_{x\in C} \psi(0,x)$ is finite since
\begin{align*}
\psi(0,x) = \EE_x[\Hpi(\tilde{\tau}_C)],
\end{align*}
which is uniformly bounded on $C$ by assumption.
\dd 

\vspace{.5cm}

Consider a sequence $(\varphi_{\epsilon})_{\epsilon > 0}$ of continuous
functions such that $\varphi_{\epsilon}(x) \downarrow \mathbf{1}_C(x)$
and $\e\leq \vpe(x) \leq 1$ for all $x \in E$. 
This is possible because $C$ is compact.
We set, for all $\epsilon > 0$, 
\[
\psi_{\epsilon}(t,x) = \EE_x \Big[\int_0^{\infty} e^{-r\int_0^s \varphi_{\epsilon}(X_u) du} (\Hpi)'(s + t) ds \Big].
 \]

\vspace{.5cm}

\textbf{Step 1: Computation of $(\partial_t + \L) \psi_{\epsilon}(t,x)$}.
We first have, for $(t,x) \in \RR_+ \times E$,
\begin{align*}
\partial_t \psi_{\epsilon}(t,x) = \EE_x \Big[\int_0^{\infty} e^{-r\int_0^s \varphi_{\epsilon}(X_u) du} (\Hpi)''(s + t) ds \Big].
\end{align*}
This is easily justified, using that $\varphi_\e\geq \mathbf{1}_C(x)$ and
that 
\begin{align*}
 \EE_x \Big[\int_0^{\infty} e^{-r\int_0^s \mathbf{1}_C(X_u) du} (\Hpi)''(s + t) ds \Big]
&=\EE_x[(\Hpi)'(\tilde{\tau}_C + t)] \\
&\leq \phi( \Hpi(t))\EE_x[\Hpi(\tilde{\tau}_C)] <\infty 
\end{align*}
by assumption. We used that 
$(\Hpi)'(s+t)=\phi(\Hpi(s+t)) \leq \phi(\Hpi(t)) \Hpi(s)$ by \eqref{huu}.

\vspace{0.5cm}
\bb 
For the second part of the computation, we use the strong generator. \dd  We fix $t\geq 0$ and recall that
$$
\L \psi_\e(t,x)=\lim_{v \to 0} \frac1v \big(\EE_x[\psi_\e(t,X_v)] - \psi_\e(t,x) \big).
$$
For $v > 0$, we have
\begin{align*}
\EE_x[\psi_\e(t,X_v)]&= \EE_x\Big( \EE_{X_v} \Big[\int_0^{\infty} e^{-r\int_0^s \vpe(X_u) du} (\Hpi)'(s + t) ds \Big]\Big)\\
&=\EE_{x} \Big[\int_0^{\infty} e^{-r \int_0^s \vpe(X_{u+v}) du} (\Hpi)'(s + t) ds \Big]\\
&=\EE_{x} \Big[\int_0^{\infty} e^{-r \int_v^{s+v} \vpe(X_{u}) du} (\Hpi)'(s + t) ds \Big].
\end{align*}
\dd 
Noting that
\begin{align*}
\int_v^{s+v} \vpe(X_u) du = \int_0^s \vpe(X_u) du - \int_0^v \vpe(X_u) du + \int_s^{s+v} \vpe (X_u) du,
\end{align*} 
we find
\begin{align*}
&\EE_x[\psi_\e(t,X_v)]-\psi_\e(t,x)\\ \qquad &=
\EE_{x}\Big[\int_0^{\infty} e^{-r \int_0^{s} \vpe(X_{u}) du} (\Hpi)'(s + t)  \Big(e^{r \int_0^v \vpe(X_u) du} e^{- r\int_s^{s+v} \vpe(X_u) du} - 1 \Big) ds \Big]. 
\end{align*}
\dd 
We easily conclude by dominated convergence, using that $\indiq_C\leq\vpe\leq 1$
and that
$$
\bb
\EE_{x}\Big[\int_0^{\infty} e^{-r \int_0^{s} \vpe(X_{u}) du} (\Hpi)'(s + t)d s\Big]
\le \EE_x[\Hpi(\tilde{\tau}_C + t)]\leq \Hpi(t)\EE_x[\Hpi(\tilde{\tau}_C)] <\infty,
$$
that
\begin{align*}
\L \psi_{\epsilon}(t,x) = &\lim \limits_{v \to 0} \frac1v\Big(\EE_x[\psi_\e(t,X_v)]-\psi_\e(t,x)\Big)\\
=& r\EE_x \Big[\int_0^{\infty} e^{-r\int_0^s \vpe(X_u) du} (\Hpi)'(s+t) (\vpe(x) - \vpe(X_s)) ds \Big] \\
=& r\vpe(x) \psi_{\epsilon}(t,x) -
r\EE_x \Big[\int_0^{\infty} e^{-r\int_0^s \vpe(X_u) du}  (\Hpi)'(s+t) \vpe(X_s) ds \Big].
\end{align*}
Note that $\partial_s (e^{-r\int_0^s \vpe(X_u) du}) = - r\vpe(X_s) e^{-r\int_0^s \vpe(X_u) du}$
a.s., so that, by integration by parts, 
\begin{align*}
r\EE_x \Big[\int_0^{\infty} e^{-r\int_0^s \vpe(X_u) du}  (\Hpi)'(s+t) \vpe(X_s) \Big] &= 
\EE_x \Big[ \big[-e^{-r\int_0^s \vpe(X_u) du}  (\Hpi)'(s+t)\big]^{\infty}_0 \Big] \\
&\quad + \EE_x \Big[\int_0^{\infty} e^{-r \int_0^s \vpe(X_u) du} (\Hpi)''(s+t) ds \Big].
\end{align*}
Using that $\vpe\geq \e$ and the properties of $\phi$ ($(\Hpi)'$ is subexponential),
one can check that
\[ \lim \limits_{s \to \infty} \EE_x \Big[  e^{-r \int_0^s \vpe(X_u) du} (\Hpi)'(s+t)  \Big]  = 0, \]
\dd
from which we conclude that
$$ \EE_x\Big[ \big[-e^{-r\int_0^s \vpe(X_u) du} (\Hpi)'(s+t) \big]^{\infty}_0 \Big] = (\Hpi)'(t)
= \phi(\Hpi(t)). $$
\dd 
We have proved that, in the sense of the strong generator (which a fortiori implies the result for the weak generator), 
\begin{align*}
 (\partial_t + \L) \psi_{\epsilon}(t,x) &= r \vpe(x) \psi_\e(t,x) - \phi(\Hpi(t)). 
\end{align*}
\bb
Note that 
\begin{align*}
\psi_{\e}(t,x) &= \EE_x \Big[\int_0^{\infty} e^{-r\int_0^s \varphi_{\epsilon}(X_u) du} (\phi \circ \Hpi)(s + t) ds \Big]  \\
&\le 2 \phi \circ \Hpi(t) \EE_x \Big[\int_0^{\infty} e^{-r\int_0^s \varphi_{\epsilon}(X_u) du} (\Hpi)'(s) ds \Big] = 2 \phi \circ \Hpi(t) \psi_{\e}(0,x),
\end{align*}
where we used \eqref{eq:ineqPhiCircH}. We thus conclude that
\begin{align}
\label{IneqPsiEps}
(\partial_t + \L) \psi_{\epsilon}(t,x) &\le 2r \phi(\Hpi(t)) \vpe(x) \psi_\e(0,x) - \phi(\Hpi(t)).
\end{align}
\dd 

\vspace{.5cm}

\textbf{Step 2: limit as $\epsilon \to 0$.}
By (\ref{IneqPsiEps}), we know that
$$
\bb
M_t^\e = \psi_{\epsilon}(t, X_t) - \psi_{\epsilon}(0,x) - 2r \int_0^t \vpe(X_s) \phi (\Hpi(s))
\psi_{\epsilon}(0,X_s) ds + \int_0^t \phi(\Hpi(s)) ds
$$
is a local supermartingale for each $\epsilon>0$, and we want to check that
$$
\bb
M_t = \psi(t, X_t) - \psi(0,x) - 2r \int_0^t \indiq_C(X_s) \phi(\Hpi(s))
\psi(0,X_s) ds + \int_0^t \phi( \Hpi(s)) ds
$$
is also a local supermartingale.

\vspace{.5cm}
It classically suffices to check that for all $T>0$, $\sup_{[0,T]}|M^\e_t-M_t|\to 0$ a.s.
as $\e\to 0$.
To this aim, the only issue is to verify that for all $T>0$, all compact subset $K\subset E$,
\begin{equation}\label{cqdf}
\sup_{[0,T]\times K}|\psi_\e(t,x)-\psi(t,x)| \to 0.
\end{equation}

Recalling that $\vpe\geq\indiq_C$ and that $(\Hpi)'$ is non-decreasing, we observe that
by definition of $\psi_\e$ and $\psi$, it holds that 
$$
\sup_{[0,T]}|\psi_\e(t,x)-\psi(t,x)|=\psi(T,x)-\psi_\e(T,x).
$$
Since now $\vpe \downarrow \mathbf{1}_C$ pointwise, we deduce from the monotone convergence
theorem that for each $x\in E$,
\begin{align*}
\psi_{\epsilon}(T,x) = \EE_x \Big[\int_0^{\infty} e^{-\int_0^s \vpe(X_u) du} (\Hpi)'(T+s) ds \Big] \overset{\epsilon \to 0}{\uparrow}
&\EE_x \Big[\int_0^{\infty} e^{-\int_0^s \mathbf{1}_C(X_u)du} (\Hpi)'(T+s) ds \Big] \\
&= \EE_x[\Hpi(\tilde{\tau}_C + T)] = \psi(T,x). 
\end{align*}
By \cite[Theorem 17.25]{Kallenberg02}, it follows from the Feller property that when
$y \to x$, the process $(X_t^y)_{t \geq 0}$ with semigroup $(\P_t)_{t \geq 0}$ and $X^y_0 = y$
converges in
distribution, in the Skorokhod space $\mathbb{D}([0,\infty),E)$,
towards the process $(X_t^x)_{t \geq 0}$ with semigroup $(\P_t)_{t \geq 0}$ and $X_0^x = x$.
We easily deduce the continuity in $x$ of $\psi_{\epsilon}(T,x)$ and $\psi(T,x)$.
We then may use Dini's theorem to conclude that, as desired,
$$
\sup_{x\in K} [\psi(T,x)-\psi_\e(T,x)] \to 0
$$
as $\e\to 0$, for any compact $K$ of $E$.

\vspace{.3cm}

\textbf{Step 3 : Conclusion.} It remains to verify that
\[ \L \psi(0,x) \leq \kappa \mathbf{1}_C(x) - \eta.\]
Using Step 1 with $t = 0$, we have
\bb
\begin{align*}
\L \psi_{\epsilon}(0,x) =& 2r \phi((\Hpi(0)) \vpe(x) \psi_{\epsilon}(0,x) - \EE_x \Big[ \int_0^{\infty} e^{- r \int_0^s \vpe(X_u) du} (\Hpi)''(s) ds \Big] - \phi(\Hpi(0))\\
\leq & 2r \phi(1) \vpe(x) \psi_{\epsilon}(0,x) - \phi(1).
\end{align*}
\dd 
We throwed away the non-negative expectation and used that $\Hpi(0)=1$.
Using the same limit procedure as in Step 2 (through local supermartingales), we conclude that
$$
\bb 
\L \psi(0,x) \leq 2r \phi(1) \indiq_C(x) \psi(0,x) -\phi(1)
$$
and conclude using that $\psi(0,x)$ is bounded on $C$.

\subsection{Proof of the result from \hyperref[C2]{(C2)}} \

\vspace{.5cm} 

\textbf{Existence of an invariant measure. } According to
\cite[Theorems 5 and 6]{SurveyMeynTweedie}, an invariant probability measure $\pi$ exists as soon as
there exist a petite set $C$, a constant $b>0$ and a continuous function
$W: E \to [0,\infty)$ such that 
\[ \L W(x) \leq -1 + b \mathbf{1}_C(x). \]
It directly follows from \hyperref[C2]{(C2)} that $W(x) := \frac{\psi(0,x)}{\eta}$ is convenient.
Moreover, by \cite[Theorem 7]{SurveyMeynTweedie}, for all $x \in E$,
\begin{align}
\label{Ergodicity}
\|\P_t(x,\cdot) - \pi(\cdot)\|_{TV} \to 0, \quad \hbox{ as } t \to \infty. 
\end{align}

\vspace{.5cm} 

\textbf{Convergence result}
By \cite[Theorem 1]{FortRoberts}, with $f_* = 1$ and $r_*(s) = \phi (\Hpi(s))$
for all $s \geq 0$, $\Psi_1(u)=u$ and $\Psi_2(v)=1$, it suffices to verify the following
three conditions.

\vspace{.5cm} 

(a) $r_*$ is a rate function in the sense of \cite{FortRoberts}, i.e.  $\lim_{s\to\infty} \frac1s\log(r_*(s))=0$. \bb This is given by Lemma \ref{lemma:subgeometric}. \dd 

\vspace{.5cm} 

(b) There is $t_0>0$ such that the Markov chain with matrix $\P_{t_0} $ is irreducible.  This follows from (\ref{Ergodicity}) and \cite[Theorem 6.1]{MTContinuous2}.

\dd

\vspace{.5cm} 

(c) There is $\delta>0$ such that, with the petite set $C$ of \hyperref[C2]{(C2)}
and recalling that $$\tau_C(\delta) = \inf \{t \geq  \delta, X_t \in C\},$$
$$
\sup_{x \in C} \EE_x \Big[ \int_0^{\tau_C(\delta)} 1 ds \Big] + \sup_{x \in C} \EE_x
\Big[ \int_0^{\tau_C(\delta)} \phi(\Hpi(s)) ds \Big] < \infty.
$$

Since $\phi$ is bounded from below, it suffices to study the second term.
By the usual supermartingale argument, recalling the condition on $\psi$, we have
\begin{align*}
\bb 
\EE_x[\psi(\tau_C(\delta),X_{\tau_C(\delta)})] \leq \psi(0,x) + \kappa \EE_x \Big[ \int_0^{\tau_C(\delta)} \mathbf{1}_C(X_s) \phi(\Hpi(s)) ds \Big] - \EE_x \Big[ \int_0^{\tau_C(\delta)} \phi( \Hpi(s)) ds \Big].
\end{align*}
\bb Since now $\int_0^{\tau_C(\delta)} \phi(\Hpi(s)) \mathbf{1}_C(X_s) ds = \int_0^{\delta} \phi(\Hpi(s)) \mathbf{1}_C(X_s) ds \leq \phi(\Hpi(\delta)) \delta$, we conclude that
\[ \EE_x \Big[ \int_0^{\tau_C(\delta)} \phi(\Hpi(s)) ds \Big] \leq \psi(0,x) + \kappa \delta \phi(\Hpi(\delta)). \]
\dd 
Since $\psi(0,\cdot)$ is bounded on $C$ by assumption, we conclude with e.g. $\delta=1$.

\bb
\section{Examples}
\label{sec:examples}

\subsection{A regeneration model with polynomial moments}

On $(\RR_+,\mathcal{B}(\RR_+))$, given a sequence $(\zeta_i)_{i \ge 1}$ of i.i.d. random variables with law $M$ on $(\RR_+, \mathcal{B}(\RR_+))$, we consider the process given, for all $t, s \ge 0$, for $X_t \in \RR_+$, by
\begin{align*}
\left\{
\begin{array}{lll}
X_{t+s} &= X_t - s \qquad &\hbox{ if } s < X_t, \\
X_{t+s} &= \zeta_1 - (s-X_t) \qquad &\hbox{ if } 0 \le s - X_t < \zeta_1, \\
X_{t+s} &= \zeta_2 - (s - (\zeta_1 + X_t)) \qquad &\hbox{ if } 0 \le s - (\zeta_1 + X_t) < \zeta_2, \\
\dots
\end{array}
\right.
\end{align*}
More rigorously, this can be written as
\begin{align*}
X_{t+s} = \mathbf{1}_{\{ X_t \le s \}} \Big(\zeta_J - \big(s -  \sum_{i=1}^{J-1} \zeta_i - X_t  \big) \Big) + (X_t - s) \mathbf{1}_{\{X_t > s\}},
\end{align*}
where $J$ is a random index depending on $X_t$, $s$, and the sequence $(\zeta_i)_{i \ge 1}$ given by
\begin{align*}
J = \inf \Big\{k \ge 1, \zeta_k - \big(s - \sum_{i=1}^{k-1} \zeta_i - X_t) \ge 0 \Big\}. 
\end{align*}

Such a process appears mainly in the study of the free-transport operator associated with the pure diffuse boundary condition in kinetic theory, see \cite[Section 2]{Kuo_2013} and can be seen as a pathological instance of a Piecewise Deterministic Markov Process, see for instance \cite{Davis18}, in the sense that jumps are predictible rather than truly random (although the size of the jumps remains random). 

We will prove first that this process satisfies our set of conditions from Theorem \ref{MainThm}, more precisely we will show that \hyperref[C1]{(C1)} is satisfied. More interestingly, we prove that the usual Lyapunov condition, \hyperref[C3]{(C3)} in Theorem \ref{MainThm} is not satisfied for this process if we impose the further requirement that $V$ should be non-decreasing. 

In the whole section we assume that $\EE[X_0^{\alpha + \e}] + \EE[\zeta_1^{\alpha + \e}] < \infty$ for some $\alpha > 1$, $\epsilon > 0$ and we will consider $\phi(x) = x^{\frac{\alpha - 1}{\alpha}}$ for all $x \ge 1$.
\bb 

\subsubsection{Proof that \hyperref[C1]{(C1)} is satisfied}

Let $C = [0,1]$, let $\tilde{\tau}_C^r$ be defined as in Theorem 2 and let $\phi(x) = x^{\frac{\alpha - 1}{\alpha}}$. The associated inverse function is defined on $\RR_+$ by $H_{\phi}^{-1}(x) = (x+1)^{\alpha}$.

Let $x \in \RR_+$, we compute
\begin{align}
\label{eq:Hpi_ex_1}
\EE_x[\Hpi(\tilde{\tau}_C^r)] \le \Hpi(1) + \EE_x[\Hpi(\tilde{\tau}_C^r) \mathbf{1}_{\{\tilde{\tau}_C^r \ge 1\}}], 
\end{align}
and one only needs to show that for all $x \in \RR_+$,
\begin{align}
\label{eq:Hpi_ex_3}
\EE_x[\Hpi(\tilde{\tau}_C^r) \mathbf{1}_{\{\tilde{\tau}_C^r \ge 1\}}] < \infty.
\end{align}
We will also show that this quantity is uniformly bounded on $C$.
Using Remark \ref{RmkHpiTauC} and that $(\Hpi)'(s) = \phi(\Hpi(s)) = s^{\alpha - 1}$ for all $s \in \RR_+$, we have 
\begin{align}
\label{eq:Hpi_ex_2}
\EE_x[\Hpi(\tilde{\tau}_C^r) \mathbf{1}_{\{\tilde{\tau}_C^r \ge 1\}}] = \EE_x \Big[\int_1^{\infty} e^{-r\int_0^s \mathbf{1}_C(X_u) du} (s+1)^{\alpha-1} ds \Big]. 
\end{align} 
First, we observe that, letting $T_0 = 0$, $T_{k+1} = T_k + X_{T_k}$ (using that $(X_t)_{t \ge 0}$ is càdlàg), for all $k \ge 1$, using the strong Markov property,
\begin{align*}
\EE_x \Big[e^{-r \int_0^{T_{k}} \mathbf{1}_C(X_u) du}\Big] \le \EE_x\Big[e^{-r \int_{T_1}^{T_2} \mathbf{1}_C(X_u) du}\Big]^{k-1} =: \mathcal{M}(r)^{k-1},
\end{align*}
where we used that, setting $\mathcal{F}_t = \sigma(X_s, s \le t)$ for all $t \ge 0$, $T_1$ is $\mathcal{F}_0$ measurable, the $\mathcal{F}_0$ random variable $e^{-r \int_0^{T_1} \mathbf{1}_C(X_u) du }$ is smaller than $1$ by positivity, and the sequence $(\zeta_i)_{i \ge 1}$ is i.i.d.. Note that the value $\mathcal{M}(r)$ is independent of $x$, since it is independent of $\mathcal{F}_{T_1-}$, and that $0 < \mathcal{M}(r) < 1$ for all $r > 0$. More precisely, since $X_{T_1} = 0$ $\PP_x$-almost surely, using the strong Markov property 
\[ \mathcal{M}(r) = \EE_0 \big[ e^{-r \int_0^{\zeta_1} \mathbf{1}_C(X_u) du} \big] = \EE \big[e^{-r (1 \wedge \zeta_1)}\big]. \]

Let $n =  [s^\eta] $, where $[y]$ denotes the integer part of $y \in \RR$, and where $\eta$ to be chosen lies in $(0,1)$. In the sequel, we omit the dependency of $n$ in $s$ for the sake of clarity. We have
\begin{align*}
 \EE_x \Big[\int_1^{\infty} e^{-r\int_0^s \mathbf{1}_C(X_u) du} (s+1)^{\alpha-1} ds \Big] &=  \EE_x \Big[\int_1^{\infty} e^{-r\int_0^{T_n} \mathbf{1}_C(X_u) du} e^{r \int_s^{T_n} \mathbf{1}_C(X_u) du} (s+1)^{\alpha-1} ds \Big] \\
 &=
 \int_1^{\infty} (s+1)^{\alpha -1} \EE_x\Big[ e^{-r\int_0^{T_n} \mathbf{1}_C(X_u) du} e^{r \int_s^{T_n} \mathbf{1}_C(X_u) du} \Big] ds,
\end{align*}
where we used Tonelli's theorem.
We now split the quantity $ e^{r \int_s^{T_n} \mathbf{1}_C(X_u) du}$ into three parts, recalling that $T_1$ is $\mathcal{F}_0$-measurable. We have, for $\gamma > 0$ to be chosen later, $\PP_x$ almost surely for all $x \in \RR_+$,
\begin{align*}
 e^{r \int_s^{T_n} \mathbf{1}_C(X_u) du} &= e^{r \int_s^{T_n} \mathbf{1}_C(X_u) du} \mathbf{1}_{\{T_n \le s \}} + e^{r \int_s^{T_n} \mathbf{1}_C(X_u) du} \mathbf{1}_{\{0 \le T_n - s \le \gamma + T_1 \}} \\
 &\qquad + e^{r \int_s^{T_n} \mathbf{1}_C(X_u) du} \mathbf{1}_{\{T_n - s \ge \gamma + T_1 \}} \\
 &\le 1 + e^{r(\gamma+T_1)} +  e^{r \int_s^{T_n} \mathbf{1}_C(X_u) du} \mathbf{1}_{\{T_n - s \ge \gamma+ T_1 \}}.
\end{align*}
Hence 
\begin{align*}
 \EE_x \Big[\int_1^{\infty} e^{-r\int_0^s \mathbf{1}_C(X_u) du} (s+1)^{\alpha-1} ds \Big] \le I_1 + I_2,
\end{align*}
where, noting that $T_1 = X_{T_0} + T_0 = x$ on $\{X_0 = x\}$, 
\[ I_1 :=  (1 + e^{r\gamma + rx}) \int_1^{\infty} (s+1)^{\alpha - 1} \EE_x\Big[e^{-r\int_0^{T_n} \mathbf{1}_C(X_u) du}\Big] ds, \]
and 
\[ I_2 := \int_1^{\infty} (s+1)^{\alpha - 1} \EE_x\Big[e^{-r\int_0^{T_n} \mathbf{1}_C(X_u) du}  e^{r \int_s^{T_n} \mathbf{1}_C(X_u) du} \mathbf{1}_{\{T_n - s \ge \gamma + T_1\}}\Big] ds. \]
For $I_1$, we immediatly obtain
\begin{align}
\label{eq:controlI_1}
I_1 \le (1 + e^{r(\gamma+x)}) \int_1^{\infty} (s+1)^{\alpha - 1} \mathcal{M}(r)^{[s^\eta] - 1} ds &\le \frac{1 + e^{r(\gamma+x)}}{\mathcal{M}(r)^2} \int_1^{\infty} (s+1)^{\alpha-1} \mathcal{M}(r)^{s^\eta} ds \\
&< \infty, \nonumber
\end{align}
where we used that $[s^\eta] + 1 \ge s^\eta$ and that $\mathcal{M}(r) < 1$. On the other hand, on $\{T_n - s \ge \gamma+T_1\}$, since $\gamma > 0$, 
\[ -\int_0^{T_n} \mathbf{1}_C(X_u) du + \int_s^{T_n} \mathbf{1}_C(X_u) du = - \int_0^s \mathbf{1}_C(X_u) du \le 0, \]
hence,
\[ I_2 \le \int_1^{\infty} (s+1)^{\alpha - 1} \PP_x(T_n-T_1  \ge \gamma + s). \]
We now claim that 
\begin{align}
\label{eq:momentT_n}
\EE_x\big[(T_n - T_1)^{\alpha + \e}\big] \le K n^{c + 1},
\end{align} 
for two constants $K, c > 0$ independent of $n$, so that, recalling $n = [s^\eta] \le s^\eta$, an application of Markov's inequality leads to 
\begin{align}
\label{eq:controlI2}
I_2 \le K \int_1^{\infty} \frac{(s+1)^{\alpha-1 + \eta (c+1)}}{(\gamma + s)^{\alpha + \e}} ds.
\end{align}
With the choice $\eta = \tfrac{\e}{2(c+1)}$, we obtain that
\begin{align*}
I_2 < \infty. 
\end{align*}
from which the conclusion of the proof follows. 
It only remains to prove \eqref{eq:momentT_n}. This follows by noting first that $T_{k+1} - T_k = \zeta_k$ for all $k \ge 1$, with $(\zeta_k)_{k \ge 1}$ i.i.d. random variables, such that $\EE[\zeta_1^{\alpha + \e}] < \infty$ by hypothesis, and also that, by properties of polynomials, there exists $K > 0$ depending only on $\alpha + \e$ such that for all $x, y \ge 0$, $(x+y)^{\alpha + \e} \le K (x ^{\alpha +\e} + y^{\alpha + \e} )$. With this at hand, the claim \eqref{eq:momentT_n} is completed by applying the same computation as in \cite[Remark 26]{bernou_fournier_collisionless_2019}.

Combining \eqref{eq:controlI_1} and \eqref{eq:controlI2}, we obtain \eqref{eq:Hpi_ex_3}, the uniform control on $C$ being a consequence of the fact that $T_1 = x \le 1$ for all $x \in C$. Gathering \eqref{eq:Hpi_ex_1} and \eqref{eq:Hpi_ex_3} leads to the desired result. 

\subsubsection{Proof that \hyperref[C3]{(C3)} of Theorem \ref{MainThm} can not be satisfied}

We prove that a strong version of (C3), in which $V$ is either non-decreasing or non-monotonous outside every compact set, can not be satisfied. It might still be the case that the condition can be satisfied for some $V$ and non-decreasing outside some compact set. Finding a corresponding set of parameters $(V,\kappa, C, \phi)$, or proving that none of those exists, is an interesting problem that we plan to address more deeply in the future, since it will bring a definite conclusion to whether (C3) can be satisfied for this model. We prove the following:

\begin{prop}
Assume that there exists $\eta > 0$ such that $[0,\eta) \subset \mathrm{supp}(f_{\zeta})^c$, where $f_{\zeta}$ is the density of $\zeta_1$ with respect to the Lebesgue measure. 
\begin{enumerate}
\item There does not exist a function $V: \RR_+ \to [1,\infty)$, continuous, increasing, such that for some compact set $C$, for some constant $\kappa > 0$, 
\begin{align}
\label{eq:Old_Lyap}
 \L V(x) \le - \phi \circ V(x) + \kappa \mathbf{1}_C(x). 
 \end{align}
\item There does not exist a function $V: \RR_+ \to [1,\infty)$ continuous with pre-compact sublevel sets such that $V$ is non-monotonuous outside every interval of the form $[0,M]$ for some $M > 0$ and such that, for some compact set $C$, for some constant $\kappa >0$, \ref{eq:Old_Lyap} holds.
 \end{enumerate}
\end{prop}

\begin{proof}
Recall that, in the sense of the extended generator, \eqref{eq:Old_Lyap} is equivalent to the fact that the process
\[ M_t = V(X_t) - V(x) - \kappa \int_0^t \mathbf{1}_C(X_u) du + \int_0^t \phi(V(X_u)) du, \]
is a $\PP_x$-local supermartingale. Let $C$ be a compact set, $\kappa > 0$ be a constant, and $\phi$ be any positive function from $[1,\infty)$ to itself. 

\begin{enumerate}
\item Let $s > 0$, $\delta \in (0,\eta)$ to be chosen and consider $x = s$, $t = s + \delta$. Then, according to the dynamic, we have
\begin{align*}
&\EE_x \big[V(X_t)\big] - V(x) - \kappa \EE_x \Big[\int_{0}^{t} \mathbf{1}_C(X_u) du\Big] + \EE_x\Big[\int_0^t \phi(V(X_u)) du \Big] \\
&\qquad \qquad \ge \EE_x \big[V(\zeta_1 - \delta)\big] - V(s) - \kappa (s + \delta)
\end{align*}
where we used that $\mathbf{1}_C(X_u) \le 1$ for all $u \in [0,t]$, and the hypothesis on the support of $f_{\zeta}$ as well as the dynamics of $(X_t)_{t \ge 0}$. For $s$ small enough, so that $s \le \tfrac{\eta}{4}$ and $\kappa s \le \tfrac{V(\tfrac{\eta}{2}) - V(\tfrac{\eta}{4})}{4}$, this last quantity being positive by hypothesis on $V$, we find
\begin{align*}
&\EE_x \big[V(X_t) \big] - V(x) - \kappa \EE_x \Big[\int_{0}^{t} \mathbf{1}_C(X_u) du\Big] + \EE_x \Big[\int_0^t \phi(V(X_u)) du \Big] \\
&\qquad \qquad \ge \EE_x \big[V(\zeta_1 - \delta) \big] - V(\tfrac{\eta}{4}) - \kappa \delta - \tfrac{V(\tfrac{\eta}{2}) - V(\tfrac{\eta}{4})}{4}
\end{align*}
and since $\zeta_1 \ge \eta$ almost surely, we have,
\begin{align*}
&\EE_x \big[V(X_t) \big] - V(x) - \kappa \EE_x \Big[\int_{0}^{t} \mathbf{1}_C(X_u) du\Big] + \EE_x\Big[\int_0^t \phi(V(X_u)) du\Big] \\
&\qquad \qquad \ge \EE_x \big[V(\eta - \delta) \big] - V(\tfrac{\eta}{4}) - \kappa \delta - \tfrac{V(\tfrac{\eta}{2}) - V(\tfrac{\eta}{4})}{4}.
\end{align*}
Choosing $\delta \le \tfrac{\eta}{2}$ small enough so that $\kappa \delta \le \tfrac{V(\tfrac{\eta}{2}) - V(\tfrac{\eta}{4})}{4}$, we find 
\begin{align*}
&\EE_x \big[V(X_t) \big] - V(x) - \kappa \EE_x \Big[\int_{0}^{t} \mathbf{1}_C(X_u) du\Big] + \EE_x \Big[\int_0^t \phi(V(X_u)) du \Big] \\
&\qquad \qquad  \ge \tfrac{V(\tfrac{\eta}{2}) - V(\tfrac{\eta}{4})}{2} > 0,
\end{align*}
where we used that $V$ is increasing. This computation holds for all $\kappa > 0$, all choice of function $\phi$ satisfying the hypothesis of Theorem \ref{MainThm} and all choice of compact set $C$, hence the result.

\item \bb Since $C$ is compact, there exists $M > 0$ such that $C \subset [0,M]$. By assumption, $V$ is non-monotonuous on $[0,M]^c$, i.e. there exists $M < y < z$ such that $V(y) > V(z)$. Let $x = z$, $s = z-y$, we have
\begin{align*}
M_{s} = V(y) - V(z) + \int_0^{z-y} \phi \circ V(X_u) du, 
\end{align*}
and it easily follows that $M_s > 0$ a.s., hence \eqref{eq:Old_Lyap} can not be satisfied. \qedhere
\end{enumerate}
\end{proof}

\subsection{Compound Poisson-process driven Ornstein-Uhlenbeck process.} 

We show how to apply \hyperref[C2]{(C2)} from Theorem \ref{MainThm} to this example, which is a classical one for the application of \eqref{ConditionLyapSubGeom}, see for instance \cite{FortRoberts, DFG}. Let $Y$ be an Ornstein-Uhlenbeck process driven by a finite rate subordinator:
\begin{align}
\label{eq:Ornstein_basic}
dY_t = - \mu Y_t dt + d Z_t, 
\end{align}
where $Z_t := \sum_{i=1}^{N_t} W_i$ with $(W_i)_{i \ge 1}$ an independent and identically distributed collection of random variables from probability measure $F$ and $(N_t)_{t \ge 0}$ is a Poisson-process of finite rate $\lambda$, independent of the collection $(W_i)_{i \ge 1}$ and $\mu$ is a positive constant. Write $(\P_t)_{t \ge 0}$ for the associated semigroup, $\L$ for the associated operator. Those processes are used as storage models (see for instance \cite{Lund_1996}), as well as for modeling the stochastic volatility (see the seminal paper \cite{Barndorff_Nielsen_2001}). We consider here the non geometrically ergodic case on which the condition corresponding to \eqref{ConditionLyapSubGeom} is applied in Fort-Roberts \cite{FortRoberts}. For this, we assume that $F(\cdot)$ is extremely heavy-tailed while still preserving the positive Harris-recurrence. For $G$ defined by $G(A) = F(e^A)$, assuming that for all $\kappa > 0$, $\int e^{\kappa x} G(dx) = + \infty$, it is proven in \cite[Lemma 17]{FortRoberts} that the process fails to be exponentially ergodic. 
We consider the polynomial case, with the paradigmatic example
\begin{align*}
F(dx) &= \frac{dx}{x^k} \mathbf{1}_{\{x \ge 1\}} 
\end{align*}
with $k > 2$ in mind. 

With a slightly stronger hypothesis (assuming a moment of order $1$ rather than only some log-moment), which allows a clearer proof, we now recover a result of the same taste as those of \cite{FortRoberts} using our new condition.   

\begin{lemma}
Consider the process $(Y_t)_{t \ge 0}$ solving \eqref{eq:Ornstein_basic}. Assume that $\EE[W_1] < \infty$. Then
\begin{align*}
\lim \limits_{t \to \infty} t^{r-1} \| \P^t(x,\cdot) - \pi(\cdot)\|_{TV} = 0. 
\end{align*}
\end{lemma}

\begin{proof}
Note that, for differentiable functions, the strong generator $\L$ of $(Y_t)_{t \ge 0}$ coincides with the weak one and is given by
\[ \L V(x) = \int_0^{\infty} \big(V(x+u) - V(x)\big) \lambda F(du) - \mu x V'(x). \]
Set, for all $t$ and $x$ large enough,
\[ \psi(t,x) = 2 \big(\ln(x) + t \mu \big)^r - \Big(\frac{t}{r} + 1 \Big)^r. \]
Clearly,
\begin{align*}
\partial_t \psi(t,x) = 2 r \mu \big(\ln(x) + t \mu \big)^{r-1} - \Big(\frac{t}{r}+1 \Big)^{r-1}. 
\end{align*} 
Moreover,
\begin{align*}
- \mu x \partial_x \psi(t,x) = - 2 \mu r \big(\ln(x) + t \mu \big)^{r-1},
\end{align*}
hence
\begin{align*}
\big( \partial_t - \mu x \partial_x \big) \psi(t,x) =  - \Big(\frac{t}{r}+1 \Big)^{r-1}. 
\end{align*}
Since $(\ln(x)+t)^r$ is concave in $x$ for $x,t$ large enough, we also have, for some positive constant $c > 0$
\begin{align*}
\int_0^{\infty} \big(\psi(t,x+u) - \psi(t,x) \big) F(du) \le c \int_0^{\infty} (\ln(x) + t \mu)^{r-1} \frac{ur}{x} F(du). 
\end{align*}
Overall, we find
\begin{align*}
(\partial_t + \L) \psi(t,x) \le cr \EE[W_1] \frac{(\ln(x) + t\mu)^{r-1}}{x} - \Big(\frac{t}{r} + 1\Big)^{r-1}. 
\end{align*}
Since all bounded sets are petite in this example, we conclude to the desired result with the choice $\phi(s) = s^{\frac{r-1}{r}}$.  
\end{proof}

It is worth noticing that the computation performed in Douc-Fort-Guillin \cite[Section 5.4]{DFG} with the same process, and which allows to derive some different subgeometric rates under suitable hypothesis on $F$, can also be adapted to fit our framework. This relies on the introduction of $\psi$ in the form 
\[ \psi(t,x) = 2 \Hpi (H_{\phi}(V(x)) + t) - \Hpi(t), \]
where $V$ and $\phi$ are the functions used in the application of \eqref{ConditionLyapSubGeom} in \cite{DFG}. We omit the tedious verification for the sake of simplicity.

\dd 
\bibliographystyle{alpha}
\bibliography{biblio}
\end{document}